\documentclass[reqno]{amsart}

\usepackage[OT1]{fontenc}
\usepackage[utf8]{inputenc}

\usepackage{mathtools}
\usepackage{subfigure,graphicx}
\usepackage{sidecap}
\usepackage{color}
\usepackage{natbib}
\usepackage{hyperref}
\hypersetup{colorlinks,urlcolor=blue,citecolor=blue}

\usepackage{threeparttable} 

\newtheorem{theorem}{Theorem}[section]

\theoremstyle{definition}
\newtheorem{definition}[theorem]{Definition}

\theoremstyle{remark}
\newtheorem{remark}[theorem]{Remark}

\numberwithin{equation}{section}

\def\Rset{\mathbb{R}}

\newcommand{\bm}[1]{\mbox{\boldmath $#1$}} 

\makeatletter
\def\tagform@#1{\maketag@@@{\color{blue}(\ignorespaces#1\unskip\@@italiccorr)}}
\makeatother

\graphicspath{{figs/}}
\DeclareGraphicsExtensions{.pdf,.jpg,.png}

\counterwithin{table}{section}
\counterwithin{figure}{section}

\begin{document}

\title{An invariant modification of the bilinear form test}

\author{\'Angelo G\'arate}
\address{Departamento de Estad\'istica, Pontificia Universidad Cat\'olica de Chile} 
\curraddr{Avenida Vicu\~na Mackena 4860, Santiago, Chile}
\email{\href{mailto:afgarate@uc.cl}{afgarate@uc.cl}}

\author{Felipe Osorio}
\email{\href{mailto:faosorios.stat@gmail.com}{faosorios.stat@gmail.com}}
\thanks{\emph{Orcid ID:} \href{https://orcid.org/0000-0002-4675-5201}{0000-0002-4675-5201} (F. Osorio), 
  \href{https://orcid.org/0000-0003-3969-5914}{0000-0003-3969-5914} (F. Crudu).}

\author{Federico Crudu}
\address{Department of Economics and Statistics, University of Siena, Italy}
\curraddr{Piazza San Francesco, 7/8 53100 Siena, Italy.}
\email{\href{mailto:federico.crudu@unisi.it}{federico.crudu@unisi.it}}


\maketitle

\begin{abstract}
The invariance properties of certain likelihood-based asymptotic tests as well as 
their extensions for \emph{M}-estimation, estimating functions and the generalized method 
of moments have been well studied. The simulation study reported in \citeauthor{Crudu:2020} 
[Econ. Lett. 187: 108885, 2020] shows that the bilinear form test is not invariant to one-to-one 
transformations of the parameter space. This paper provides a set of suitable conditions 
to establish the invariance property under reparametrization of the bilinear form test for 
linear or nonlinear hypotheses that arise in extremum estimation which leads to a simple 
modification of the test statistic. Evidence from a Monte Carlo simulation experiment 
suggests good performance of the proposed methodology.
\end{abstract}

\section{Introduction}\label{sec:Intro}

Recently, \cite{Dufour:2017} have studied the invariance properties of various testing
procedures for nonlinear hypotheses in the context of $M$-estimation, inference functions
and the generalized method of moments (GMM) to equivalent reformulation and reparametrization 
of the hypotheses of interest. The main conclusion of their study is that as in the likelihood 
framework \citep[see][]{Dagenais:1991}, the Lagrange multiplier-type statistics, $C(\alpha)$-type 
and the distance metric criterion proposed by \cite{Newey:1987} are invariant to reformulations 
of the null hypothesis and to one-to-one transformations of the parameter space, whereas the 
Wald-type statistic does not possess such properties. Although the Wald-type statistic cannot 
be recommended in such circumstances, some solutions \citep{Kemp:2001} have been proposed to 
make the test invariant to general reparametrizations.

It is worth emphasizing that the estimation mechanisms considered by \cite{Dufour:2017} belong
to the class of extremum estimators \citep[see, for instance][]{Gourieroux:1995, Hayashi:2000}.
Besides, they only consider test statistics that correspond to extensions of the so-called 
``holy trinity" \citep[see][]{Muggeo:2014}. By contrast, \cite{Crudu:2020} proposed a competitive 
procedure for testing nonlinear hypotheses, based on a bilinear form (BF) statistic, in the 
context of extremum estimation that generalizes the gradient statistic of \cite{Terrell:2002}. 
As stated in \cite{Rao:2005}, additional research is needed to better understand the behavior 
of the gradient statistic. This has prompted a series of works regarding various aspects of 
the asymptotic behavior of the gradient statistic \cite[see][and references therein]{Lemonte:2013}. 
The simulation study in \cite{Crudu:2020} revealed that the BF test lacks invariance under 
equivalent null hypotheses. The main aim of this work is to provide suitable conditions 
to ensure the invariance under reparametrization property of the BF test based on extremum 
estimation.

This paper is organized as follows. In Section \ref{sec:BF}, the bilinear form test for
extremum estimation is revisited. Section \ref{sec:invariance} provides sufficient conditions 
for the invariance of the bilinear form test statistic. In Section \ref{sec:experiments},
we report our numerical findings that allow us to illustrate the invariance to one-to-one 
transformations of the parameter space, i.e., reparametrizations. Finally some concluding 
remarks are discussed in Section \ref{sec:conclusion}.

\section{The bilinear form test}\label{sec:BF}

The general class of estimating procedures known as extremum estimation \citep{Gourieroux:1995} 
is defined by optimizing an objective function $Q_n(\bm{\theta})$, which depends on observed 
data $\bm{z}_1,\dots,\bm{z}_n$ and a parameter vector $\bm{\theta}\in\Theta\subset\Rset^p$, which 
includes, for example, GMM \citep{Hansen:1982}, $M$-estimation \citep{Huber:1981}, and maximum 
likelihood estimation under model misspecification \citep{White:1982}. In \cite{Crudu:2020}, 
the BF statistic was introduced to test nonlinear hypotheses of the form,
\begin{equation}\label{eq:H0}
  H_0: \bm{g}(\bm{\theta}) = \bm{0}, \qquad \text{against} \qquad H_1: \bm{g}(\bm{\theta})
  \neq \bm{0},
\end{equation}
where $\bm{g}:\Theta\to \Rset^q$ is a continuously differentiable function of $\bm{\theta}$
and $\bm{G}(\bm{\theta}) = \partial\bm{g}(\bm{\theta})/\partial\bm{\theta}^\top$ is
a $q\times p$ matrix with full row rank. Define $\bm{A}_n(\bm{\theta}) = \partial^2 
Q_n(\bm{\theta})/\partial\bm{\theta}\partial\bm{\theta}^\top$. The following regularity 
conditions are assumed:
\begin{itemize}
	\item[] \emph{A1.} $\bm{A}_n(\bm{\theta})\stackrel{\sf a.s.}{\longrightarrow}\bm{A}$ uniformly 
  in $\bm{\theta}$ with $\bm{A}$ nonsingular matrix;
	\item[] \emph{A2.} $\sqrt{n}\,\partial Q_n(\bm{\theta})/\partial\bm{\theta} \stackrel{\sf D}
  {\longrightarrow}\mathsf{N}_p(\bm{0},\bm{B})$.
\end{itemize}
Then, the BF statistic for testing the hypothesis defined in \eqref{eq:H0} is given by
\begin{equation}\label{eq:BF}
  BF_n(\bm{g}) = n\Big\{\frac{\partial Q_n(\widetilde{\bm{\theta}}_n)}{\partial\bm{\theta}}\Big\}^\top
  \bm{G}^+\bm{S\Omega}^{-1}\bm{g}(\widehat{\bm{\theta}}_n),
\end{equation}
where $\bm{G} = \bm{G}(\bm{\theta})$, $\bm{S} = \bm{G}(-\bm{A})^{-1}\bm{G}^\top$, $\bm{\Omega} 
= \bm{GA}^{-1}\bm{BA}^{-1}\bm{G}^\top$ with $\bm{G}^+ = \bm{G}^\top (\bm{GG}^\top)^{-1}$ being 
the Moore-Penrose inverse of $\bm{G}$, $\widehat{\bm{\theta}}_n$ is the unrestricted extremum 
estimator and the constrained estimator $\widetilde{\bm{\theta}}_n$ is the solution to the problem:
\[
  \max_{\theta\in\Theta}\ Q_n(\bm{\theta}), \qquad \text{subject to: $\bm{g}(\bm{\theta})
  =\bm{0}$}.
\]
Given Assumptions \textit{A1}-\textit{A2} and under $H_0$, the $BF_n(\bm{g})$ statistic
given in (\ref{eq:BF}) asymptotically has a chi-square distribution with $q$ degrees of
freedom. Details about the asymptotic distribution for the BF statistic and its asymptotic
equivalence with the Lagrange multiplier test statistic can be found in \cite{Crudu:2020}.

\smallskip

\begin{remark}
  The test statistic $BF_n(\bm{g})$ in Equation \eqref{eq:BF} is not feasible, as matrices $\bm{G}^+$, 
  $\bm{S}$ and $\bm{\Omega}$ are typically unknown. To make $BF_n(\bm{g})$ feasible we can 
  replace those matrices with consistent estimators, say, $\widetilde{\bm{G}}{}^+$, $\widetilde{\bm{S}}$ 
  and $\widetilde{\bm{\Omega}}$, evaluated at $\widetilde{\bm{\theta}}_n$.
\end{remark}

\section{An invariant bilinear form test}\label{sec:invariance}

Let us consider a reparameterization of the parameter space \citep[see, for instance][]{Dufour:2017}
defined by a one-to-one differentiable transformation $\bm{\phi}:\Theta\to\Theta_*$ with
$\Theta\subset\Rset^p$ and $\Theta_*\subset\Rset^p$ such that $\bm{\phi}(\bm{\theta}) =
\bm{\theta}_*$ and the inverse function of $\bm{\phi}$ satisfies $\bm{\phi}^{-1}(\bm{\theta}_*)
= \bm{\theta}$. Suppose that the following holds
\begin{equation}\label{eq:condition}
  \bm{g}_*(\bm{\theta}_*) = \bm{g}(\bm{\phi}^{-1}(\bm{\theta}_*)),
\end{equation}
then, the null hypotheses $H_0:\bm{g}(\bm{\theta}) = \bm{0}$ and $H_0^*:\bm{g}_*(\bm{\theta}_*)
= \bm{0}$ are considered equivalent representations of the same hypothesis, provided that
$\bm{g}(\bm{\theta}) = \bm{0}$ if and only if $\bm{g}_*(\bm{\theta}_*) = \bm{0}$, in which case
we say that the test is invariant to reparametrization \citep{Dufour:2017}. The following
theorem states sufficient conditions for the invariance to reparametrization for the BF test.

\smallskip

\begin{theorem}\label{thm:1}
  Let $\bm{g}:\Theta_*\to \Rset^q$ be a continuously differentiable function in $\bm{\theta}_*
  \in\Theta_*$ such that $\bm{g}_*(\bm{\phi}(\bm{\theta})) = \bm{0}$ if only if $\bm{g}(\bm{\theta})
  =\bm{0}$. Let us consider the following assumptions:
  \begin{itemize}
    \item[] B1. $\bm{G}_*(\bm{\theta}_*) = \bm{G}(\bm{\theta})\bm{K}(\bm{\phi}(\bm{\theta}))$,
    where $\bm{G}_*(\bm{\theta}_*) = \partial\bm{g}_*(\bm{\theta}_*)/\partial\bm{\theta}_*^\top$,  
    $\bm{K}(\bm{\theta}_*) = \partial\bm{\phi}^{-1}(\bm{\theta}_*)/\partial\bm{\theta}_*^\top$;
    \item[] B2. $[\bm{G}(\bm{\theta})\bm{K}(\bm{\phi}(\bm{\theta}))]^+ =
    \bm{K}^+(\bm{\phi}(\bm{\theta}))\bm{G}^+(\bm{\theta})$;
    \item[] B3. $\partial Q_n(\bm{\theta}_*)/\partial\bm{\theta}_* = \bm{K}^\top(\bm{\phi}(\bm{\theta}))
    \,\partial Q_n(\bm{\theta})/\partial\bm{\theta}$;
    \item[] B4. $\bm{A}_*(\bm{\theta}_*) = \bm{K}^\top(\bm{\phi}(\bm{\theta}))\bm{A}
    \bm{K}(\bm{\phi}(\bm{\theta}))$;
    \item[] B5. $\bm{B}_*(\bm{\theta}_*) = \bm{K}^\top(\bm{\phi}(\bm{\theta}))\bm{B}
    \bm{K}(\bm{\phi}(\bm{\theta}))$;
    \item[] B6. $\bm{g}_*(\bm{\theta}_*) = \bm{g}(\bm{\phi}^{-1}(\bm{\theta}_*))$.
  \end{itemize}
  Then, the bilinear form test statistic given in Equation (\ref{eq:BF}) is invariant,
  i.e. $BF_n(\bm{g}_*) = BF_n(\bm{g})$.
\end{theorem}

\begin{proof}
  For simplicity of notation we define $\bm{G} = \bm{G}(\bm{\theta})$ and $\bm{K} = \bm{K}(\bm{\phi}(\bm{\theta}))$. 
  Using assumptions \textit{B1} and \textit{B4} we obtain
  \begin{align*}
    \bm{S}_* & = \bm{G}_*(\bm{\theta}_*)(-\bm{A}_*(\bm{\theta}_*))^{-1}\bm{G}_*^\top(\bm{\theta}_*) 
    = \bm{GK}(-\bm{K}^\top\bm{AK})^{-1}(\bm{GK})^\top \\
    & = \bm{GK}\bm{K}^{-1}(-\bm{A})^{-1}(\bm{K}^\top)^{-1}\bm{K}^\top\bm{G}^\top 
    = \bm{G}(-\bm{A})^{-1}\bm{G}^\top = \bm{S}.
  \end{align*}
  Moreover, by assumptions \textit{B1} and \textit{B5}
  \begin{align*}
    \bm{\Omega}_* & = \bm{G}_*(\bm{\theta}_*)\bm{A}_*^{-1}(\bm{\theta}_*)\bm{B}_*(\bm{\theta}_*)
    \bm{A}_*^{-1}(\bm{\theta}_*)\bm{G}_*^\top(\bm{\theta}_*) \\
    & = \bm{GK}(\bm{K}^\top\bm{AK})^{-1}\bm{K}^\top\bm{BK}(\bm{K}^\top\bm{AK})^{-1}(\bm{GK})^\top \\
    & = \bm{GKK}^{-1}\bm{A}^{-1}(\bm{K}^\top)^{-1}\bm{K}^\top\bm{BKK}^{-1}\bm{A}^{-1}(\bm{K}^\top)^{-1}
    \bm{K}^\top\bm{G}^\top \\
    & = \bm{GA}^{-1}\bm{BA}^{-1}\bm{G}^\top = \bm{\Omega}.
  \end{align*}
  Finally, with assumptions \textit{B2}, \textit{B3}, and \textit{B6} we obtain,
  \begin{align*}
    BF_n(\bm{g}_*) & = n\Big\{\frac{\partial Q_n(\widetilde{\bm{\theta}}_*)}{\partial\bm{\theta}_*}\Big\}^\top
    \bm{G}_*^+(\bm{\theta}_*)\bm{S}_*\bm{\Omega}_*^{-1}\bm{g}_*(\bm{\theta}_*) \\ 
    & = n\Big\{\frac{\partial Q_n(\widetilde{\bm{\theta}}_n)}{\partial\bm{\theta}}\Big\}^\top
    \bm{K}(\bm{GK})^+\bm{S}\bm{\Omega}^{-1}\bm{g}(\bm{\phi}^{-1}(\bm{\theta}_*)) \\
    & = n\Big\{\frac{\partial Q_n(\widetilde{\bm{\theta}}_n)}{\partial\bm{\theta}}\Big\}^\top
    \bm{K}\bm{K}^+\bm{G}^{+}\bm{S}\bm{\Omega}^{-1}\bm{g}(\bm{\theta}) \\
    & = n\Big\{\frac{\partial Q_n(\widetilde{\bm{\theta}}_n)}{\partial\bm{\theta}}\Big\}^\top
    \bm{G}^+\bm{S}\bm{\Omega}^{-1}\bm{g}(\bm{\theta}) = BF_n(\bm{g}),
  \end{align*}
  as desired.
\end{proof}

Clearly, conditions \textit{B4} and \textit{B5} can be disregarded when $\bm{B} =
-\bm{A}$ which occurs, for example, when the objective function $Q_n(\bm{\theta})$
is based on a quasi-score estimating equation. 
This is particularly relevant in the context of GMM, where it is beneficial to use 
an invariant test that avoids computing matrices $\bm{A}$ and $\bm{B}$.

It is worth noting that \cite{Greville:1966} highlights the conditions necessary to 
validate assumptions like the one stated in \textit{B2}. To establish \textit{B2}, 
it is sufficient to verify the following identity \citep[see Theorem 2 of][]{Greville:1966},
\begin{equation}\label{eq:Greville-T2a}
  \bm{G}^+\bm{GKK}^\top = \bm{KK}^\top\bm{G}^+\bm{G}.
\end{equation}
Here, $\bm{G}$ is a full row rank matrix, which implies that its Moore-Penrose inverse 
is given by $\bm{G}^+ = \bm{G}^\top(\bm{GG}^\top)^{-1}$. This expression facilitates the 
verification of identity \eqref{eq:Greville-T2a}. Additionally, $\bm{K}$ is assumed to 
be a full-rank matrix, thus we have $\bm{KK}^+ = \bm{KK}^{-1} = \bm{I}$ which can be used 
to demonstrate the symmetry of the product $\bm{G}^\top\bm{GKK}^+$. Interestingly, Assumption 
\textit{B2} is specific to the BF test, while Assumptions \textit{B4} and \textit{B5} 
are in common with the $C(\alpha)$-type statistic and, are also required for Lagrange 
multiplier-type tests \citep[see][]{Dufour:2017}. The set of conditions of Theorem \ref{thm:1} 
allows us to introduce the following correction of the bilinear form statistic that is 
invariant to the definition of the null hypothesis.

\smallskip

\begin{definition}
  Under the conditions established in Theorem \ref{thm:1}, we can define a corrected version 
  of the $BF$ statistic as:
  \begin{equation}\label{eq:corrected}
    BF_n^c = n\Big\{\frac{\partial Q_n(\widetilde{\bm{\theta}}_*)}{\partial\bm{\theta}_*}\Big\}^\top
    \bm{G}_*^+\bm{S}_*\bm{\Omega}_*^{-1}\bm{g}_*(\widehat{\bm{\theta}}_*),
  \end{equation}
  where $\bm{G}_* = \bm{G}_*(\bm{\theta}_*)$, $\bm{S}_* = \bm{G}_*(-\bm{A}_*)^{-1}\bm{G}_*^\top$, 
  $\bm{\Omega}_* = \bm{G}_*\bm{A}_*^{-1}$ $\bm{B}_*\bm{A}_*^{-1}\bm{G}_*^\top$. 
\end{definition}

\begin{table*}[!htp]
  \begin{center}
  \caption{Empirical size of $5\%$ test. $W$, $BF$ and $W^*$, $BF^*$ refer to that the 
  Wald and BF statistics are computed using the null hypotheses in Equations (\ref{eq:H0A}) 
  and (\ref{eq:H0B}), respectively.}\label{tab:exp1}
  \begin{tabular}{rrccccccc} \hline
    $k$ & $n$  & $W$   & $W^*$ & $BF$  & $BF^*$ & $BF^c$ & $LM$ & $D$  \\ \hline
     5  &   25 & 0.058 & 0.246 & 0.058 & 0.184 & 0.058 & 0.058 & 0.058 \\
        &   50 & 0.049 & 0.216 & 0.049 & 0.154 & 0.049 & 0.049 & 0.049 \\
        &  100 & 0.049 & 0.155 & 0.049 & 0.108 & 0.049 & 0.049 & 0.049 \\
        &  500 & 0.051 & 0.084 & 0.051 & 0.065 & 0.051 & 0.051 & 0.051 \\\hline
     2  &   25 & 0.058 & 0.127 & 0.058 & 0.203 & 0.216 & 0.058 & 0.058 \\
        &   50 & 0.049 & 0.104 & 0.049 & 0.073 & 0.049 & 0.049 & 0.049 \\
        &  100 & 0.049 & 0.074 & 0.049 & 0.065 & 0.049 & 0.049 & 0.049 \\
        &  500 & 0.051 & 0.058 & 0.051 & 0.055 & 0.051 & 0.051 & 0.051 \\\hline
    -2  &   25 & 0.058 & 0.155 & 0.058 & 0.137 & 0.080 & 0.058 & 0.058 \\
        &   50 & 0.049 & 0.135 & 0.049 & 0.085 & 0.036 & 0.049 & 0.049 \\
        &  100 & 0.049 & 0.103 & 0.049 & 0.066 & 0.042 & 0.049 & 0.049 \\
        &  500 & 0.051 & 0.058 & 0.051 & 0.050 & 0.051 & 0.051 & 0.051 \\\hline
    -5  &   25 & 0.058 & 0.234 & 0.058 & 0.163 & 0.058 & 0.058 & 0.058 \\
        &   50 & 0.049 & 0.216 & 0.049 & 0.143 & 0.049 & 0.049 & 0.049 \\
        &  100 & 0.049 & 0.171 & 0.049 & 0.108 & 0.049 & 0.049 & 0.049 \\
        &  500 & 0.051 & 0.096 & 0.051 & 0.060 & 0.051 & 0.051 & 0.051 \\\hline
  \end{tabular}
  \end{center}
\end{table*}

\section{Monte Carlo simulations}\label{sec:experiments}

Based on the simulation study by \cite{Lafontaine:1986}, \cite{Goh:1996} examined the small-sample 
performance of corrections to the Wald statistic proposed by \cite{Phillips:1988}. These corrections, 
based on Edgeworth expansions, aim to accelerate the convergence of the test statistic to its asymptotic 
distribution under two algebraically equivalent null hypotheses, namely: 
\begin{align}
  H_0: {} & g(\gamma,\beta) = \beta - 1 = 0, \label{eq:H0A} \\
  H_0^*: {} & g_*(\gamma_*,\beta_*) = \beta_*^k - 1 = 0, \label{eq:H0B}
\end{align}
where $k$ is a non-zero integer, and $\gamma$, $\beta$ are regression coefficients in the following 
model:
\begin{equation}\label{eq:model}
  \bm{y} = \bm{x}_1\gamma + \bm{x}_2\beta + \bm{\epsilon},
\end{equation}
with $\bm{x}_1$, $\bm{x}_2$ denoting $n\times 1$ vector of covariates and $\bm{\epsilon}$ 
representing an $n\times 1$ vector of random disturbances.

Next, we examine the performance of the proposed correction for the BF test statistic by 
carrying out a comparison against the Wald $(W)$, Lagrange multiplier $(LM)$, bilinear 
form $(BF)$, and metric distance $(D)$ statistics considering the equivalent hypotheses 
given in \eqref{eq:H0A} and \eqref{eq:H0B}. For definitions of each of these statistics 
in the context of extremum estimation refer to \cite{Dufour:2017}, \cite{Crudu:2020} and 
\cite{Newey:1987}, respectively.

\subsection{Simulation setup}

Following the simulation study reported by \cite{Goh:1996}, we considered model given in 
\eqref{eq:model} with $\bm{\theta} = (\gamma,\beta)^\top = (1,1)^\top$ and hypotheses defined 
in \eqref{eq:H0A} and \eqref{eq:H0B} for $k \in \{-5,-2,2,5\}$. $10\,000$ were generated 
with sample sizes $n=25, 50, 100$ and $500$ from model \eqref{eq:model} where $\bm{x}_j = 
(x_{1j},\dots,x_{nj})^\top$ with $x_{ij}\sim\mathsf{U}(0, 1)$ for $i=1,\dots,n$; $j=1,2$, and
$\bm{\epsilon}\sim\mathsf{N}_n(\bm{0}, \sigma^2\bm{I})$, where $\sigma^2 = 1$.

For the hypotheses $H_0$ and $H_0^*$ defined in \eqref{eq:H0A} and \eqref{eq:H0B}, respectively, 
we obtain that the reparameterization $\bm{\phi}(\bm{\theta}) = \bm{\phi}(\gamma,\beta)$, with 
\[
  \bm{\phi}^{-1}(\gamma_*,\beta_*) = (\gamma_*, \beta_*^k)^\top,
\]
satisfies the condition \textit{B6}. In accordance with hypothesis \eqref{eq:H0B} we obtain
\[
  \bm{K}(\bm{\theta}_*) = \begin{pmatrix}
    1 & 0 \\
    0 & k\beta_*^{k-1}
  \end{pmatrix}, \qquad \bm{G}_*^+(\bm{\theta}_*) = \begin{pmatrix}
    0 \\
    k^{-1}\beta_*^{1-k}
  \end{pmatrix},
\]
and the verification of conditions \textit{B1} and \textit{B2} is straightforward. In this 
example we have that $\bm{B} = -\bm{A}$, which yields to a simplified version of the BF 
statistic. Thus, only conditions \textit{B1}, \textit{B3} and \textit{B6} are required to 
obtain the corrected test statistic given in Equation \eqref{eq:corrected}. 

\subsection{Simulation results}

In the Monte Carlo experiment a $5\%$ level test was performed.\footnote{The replication files 
related to this article are available online at \url{https://github.com/faosorios/BF_invariance}} Table \ref{tab:exp1} displays 
the results of testing $H_0$ and $H_0^*$ for four values of $k$. As expected all statistics 
are in agreement with the linear hypothesis in \eqref{eq:H0A}, whereas the Wald statistic and 
the uncorrected version for the BF statistic show obvious invariance problems with empirical 
sizes that approach the nominal level as the sample size increases. The results suggest that 
the proposed correction (reported as $BF^c$ in Table \ref{tab:exp1}) allows the empirical size 
of the BF statistic to be consistently closer to the nominal value, and except for some cases 
when the sample size is small, e.g. for $k = 2$ and $n = 25$. It is worth noting that this test 
aligns with the metric distance and Lagrange multiplier statistics, $D$ and $LM$ which are known 
to be invariant.

\section{Concluding remarks}\label{sec:conclusion}

In this paper we study conditions that must be imposed on the BF statistic to avoid the 
undesirable property of non-invariance to reformulations of the hypothesis of interest 
generated by reparametrizations. Although the results of the simulation study seem promising, 
it should be noted that the required conditions are not always satisfied. For example, 
for the problem introduced by \cite{Gregory:1985} with two equivalent null hypotheses, 
\[
  H_0: \theta_1 - 1/\theta_2 = 0, 
  \qquad \textrm{and} \qquad 
  H_0^* : \theta_{*1}\theta_{*2} - 1 = 0, 
\]
we may note that the reparametrization $\bm{\phi}(\theta_1,\theta_2)$, with 
\[
  \bm{\phi}^{-1}(\theta_{*1},\theta_{*2}) = ((\theta_{*1} + 1)\theta_{*2} - 1,1/\theta_{*2})^\top,
\] 
satisfies the condition in \eqref{eq:condition}, but it does not satisfy Assumption \textit{B2}. 
Thus, as with the Wald test, caution is required in the use of the BF statistic. The Monte Carlo 
experiment in Section \ref{sec:experiments} is insightful, as it highlights the deterioration of 
the uncorrected BF statistic as nonlinearity increases. In light of the above results, further 
research is needed to better understand the properties and finite sample behavior of the BF 
statistic. A promising avenue of research is the application of the corrections based on the 
asymptotic expansions developed in \cite{Phillips:1988}.

\section*{Acknowledgements}

This work was written while the second author was at Universidad T\'ecnica Federico Santa Mar\'ia.
The authors acknowledge the support of the computing infrastructure of the Applied 
Laboratory of Spatial Analysis UTFSM\,-\,ALSA (MT7018). \'Angelo G\'arate was supported 
by the National Scholarship Program of Chile, ANID-Chile. Federico Crudu's research 
benefited from financial support provided by the University of Siena via the F-CUR 
grant 2274-2022-CF-PSR2021-FCUR 001.


\bigskip

\end{document}